\documentclass[12pt]{amsart}
\usepackage{amssymb,amsfonts,amsthm}
\usepackage{amsmath,amsthm,amssymb}
\usepackage{amscd}
\usepackage{amsfonts}
\usepackage{color}
\newcommand{\lin}{\hbox{Lin}}
\newcommand{\non}{\hbox{Non}}
\newcommand{\iso}{\hbox{Isot}}
\newcommand{\con}{\hbox{Cont}}
\newcommand{\ocs}{\hbox{OccSet}}
\newcommand{\var}{\hbox{var}}

\newtheorem{theorem}{Theorem}[section]
\newtheorem{ex}[theorem]{Example}
\newtheorem{cor}[theorem]{Corollary}

\newtheorem{fact}[theorem]{Fact}
\newtheorem{lemma}[theorem]{Lemma}
\newtheorem{definition}[theorem]{Definition}

\newtheorem{prop}[theorem]{Proposition}

\newtheorem{question}{Question}

\begin{document}

\title{The finite basis problem for words with at most two non-linear variables}
\author{ Olga Sapir }

\address{Nashville,TN,USA}
\email{olga.sapir@gmail.com}
\date{}

\begin{abstract}  Let $\mathfrak A$ be an alphabet and $W$ be a set of words in the free monoid ${\mathfrak A}^*$. Let $S(W)$ denote the Rees quotient  over the ideal of  ${\mathfrak A}^*$ consisting of all words that are not subwords of words in $W$. We call a set of words $W$ {\em finitely based}  if the monoid $S(W)$ is finitely based.

We find a simple algorithm that recognizes finitely based words among words with at most two non-linear variables. We also describe syntactically all hereditary finitely based monoids of the form $S(W)$.
\end{abstract}

\maketitle

\section{Introduction}

An algebra  is said to be {\em finitely based} (FB) if there is a finite subset of its identities from which all of its identities may be deduced.
Otherwise, an algebra is said to be {\em non-finitely based} (NFB).
The famous Tarski's Finite Basis Problem asks if there is an algorithm to decide when a finite algebra is finitely based.
In 1996, R. McKenzie  \cite{RM} solved this problem in the negative showing that the classes of FB and inherently not finitely based finite algebras are recursively inseparable. (A locally finite algebra is said to be {\em inherentely not finitely based} (INFB) if any locally finite variety containing it is NFB.)

It is still unknown whether the set of FB finite semigroups is recursive although a very large volume of work is devoted to this problem (see
the surveys \cite{SV,MV}). In contrast with McKenzie's result, a powerful description of the INFB finite semigroups has been obtained by M. Sapir \cite{MS, MS1}.
 These results show that we need to concentrate on NFB finite semigroups that are not INFB.

In 1968, P. Perkins \cite{P} found the first two examples of finite NFB semigroups.  One of these examples was the 25-element monoid obtained from the set of words
$W= \{abtba, atbab, abab, aat\}$ by using the following construction attributed to Dilworth.

 Let ${\mathfrak A}$ be an alphabet and $W$ be a set of words in the free monoid ${\mathfrak A}^*$. Let $S(W)$ denote the Rees quotient  over the ideal of  ${\mathfrak A}^*$ consisting of all words that are not subwords of words in $W$. For each set of words $W$, the semigroup $S(W)$ is a monoid with zero whose nonzero elements are the subwords of words in $W$. Evidently, $S(W)$ is finite if and only if $W$ is finite.

It is clear from the results of \cite{MS, MS1} that a finite monoid of the form $S(W)$ is never INFB.
It is shown in \cite{JS} that, with respect to the finite basis problem, the class of monoids of the form $S(W)$, shares all of the currently known bad  properties held by the class of all finite semigroups. In particular, the set of FB semigroups and the set of NFB semigroups in this class are not closed under taking direct products, and there exists an infinite chain of varieties generated by such semigroups where FB and NFB varieties alternate. Recently, a certain finite monoid of the form $S(W)$ (see Theorem 3.2 in \cite{CHLS}) emerged as the one responsible for the non-finite basis property of such infinite semigroups as the bicyclic monoid \cite{Pastin, Sh} and the monoid of $2 \times 2$ upper triangular tropical matrices.

We call a set of words $W$ {\em finitely based} if the monoid $S(W)$ is finitely based. In this paper  we study the following problem.

\begin{question}\cite[M. Sapir]{SV} \label{qMS} Is the set of finite finitely based sets of words recursive?
\end{question}

A partial answer to Question \ref{qMS} is contained in \cite[Theorem 5.1]{OS}. That theorem says that a word $\bf u$ in a two-letter alphabet $\{a,b\}$ is FB if and only if $\bf u$ is of the form $a^nb^m$ or $a^nba^m$ for some $n,m\ge 0$ modulo renaming $a$ and $b$.
If a variable $t$ occurs exactly once in a word ${\bf u}$ then we say that $t$ is {\em linear} in ${\bf u}$. If a variable $x$ occurs more than once in a word ${\bf u}$ then we say that $x$ is {\em non-linear} in ${\bf u}$.
In this article, we generalize Theorem 5.1 in \cite{OS} into an algorithm which given a word $\bf U$ with at most two non-linear variables, decides whether $\bf U$ is
finitely based or not.

A word $\bf u$ is said to be an {\em isoterm} for a semigroup $S$ if $S$ does not satisfy any nontrivial identity of the form ${\bf u} \approx {\bf v}$.
The notion of an isoterm was introduced by Perkins in \cite{P} and has proved to be crucial for understanding the difference between finitely based and non-finitely based semigroups. According to \cite{MS}, a finite semigroup $S$ is INFB  iff every Zimin word (${\bf Z}_1=x_1, \dots, {\bf Z}_{k+1} = {\bf Z}_kx_{k+1}{\bf Z}_k, \dots$)  is an isoterm for $S$ iff the word ${\bf Z}_k$ is an isoterm for $S$ where $k=|S|^2$.
If $S$ is a finite aperiodic semigroup with central idempotents then according to \cite{MJ}, every subvariety of $S$ is finitely based if and only if the word ${\bf Z}_2=xtx$ is not an isoterm for $S$.

It is not a surprise that the notion of an isoterm plays a crucial role in this article as well. In Theorem \ref{main}, we prove that a word $\bf U$ with at most two non-linear variables is FB if and only if certain words are isoterms for $S(\{{\bf U}\})$ and certain words are not.
In Theorem \ref{main2}, we present our algorithm in a computation-free form.
This work was inspired by the article \cite{LZL} where all finitely based words with two non-linear 2-occurring variables are described.

\section{A quasi-order on sets of words and how to check that a monoid of the form $S(W)$ satisfies a balanced identity}

Throughout this article, elements of a countable alphabet $\mathfrak A$ are called {\em variables} and elements of the free monoid $\mathfrak A^*$ are called {\em words}. 
We use $\var S$ to denote the variety generated by a semigroup $S$ and $\operatorname{var} \Sigma$ to denote the variety defined by a set of identities $\Sigma$.

\begin{lemma} \label{prec}  \cite[Lemma 3.3]{MJ}
Let $W$ be a set of words and $S$ be a monoid.
Then each word in $W$ is an isoterm for $S$ if and only if $\operatorname{var}(S)$ contains $S(W)$.
\end{lemma}

If $W$ and $W'$ are two sets of words then we write $W \preceq W'$ if for any monoid $S$ each word in $W'$ is an isoterm for $S$ whenever  each word in $W$ is an isoterm for $S$. It is easy to see that the relation $\preceq$ is reflexive and transitive, i.e. it is a {\em quasi-order} on sets of words. If $W \preceq W' \preceq W$ then we write $W \sim W'$.
We say that two sets of words $W$ and $W'$ are equationally equivalent if the monoids $S(W)$ and $S(W')$ satisfy the same identities.
The following proposition shows that if we identify sets of words modulo $\sim$ then we obtain an ordered set antiisomorphic to
the set of all varieties of the form $\var S(W)$ ordered under inclusion. In particular, two sets of words $W$ and $W'$ are equationaly equivalent if and only if $W \sim W'$.

\begin{prop} \label{pr1} For two sets of words $W$ and $W'$
the following conditions are equivalent.

(i)  $W \preceq W'$.

(ii) Each word in $W'$ is an isoterm for $S(W)$.

(iii) $\operatorname{var} S(W)$ contains  $S(W')$.

\end{prop}

\begin{proof} (i) $\rightarrow$ (ii)  Since  each word ${\bf w} \in W$ is an isoterm for $S(W)$,   each word ${\bf w'} \in W'$ is also an isoterm for $S(W)$.

 (ii) $\rightarrow$ (iii) Since  each word ${\bf w'} \in W'$ is an isoterm for $S(W)$, Lemma \ref{prec} implies that the variety generated by  $S(W)$ contains  $S(W')$.

 (iii) $\rightarrow$ (i) Let $S$ be a monoid such that each word ${\bf w} \in W$ is an isoterm for $S$. Then by Lemma \ref{prec} $\operatorname{var}(S)$ contains $S(W)$.
 Since the variety generated by  $S(W)$ contains  $S(W')$, $\operatorname{var}(S)$ contains $S(W')$. Therefore,
 by Lemma \ref{prec} each word ${\bf w'} \in W'$ is an isoterm for $S$.
\end{proof}

 The relations $\preceq$ and $\sim$ can be extended to individual words. For example, if $\bf u$ and $\bf v$ are two words then ${\bf u} \sim {\bf v}$ means $\{{\bf u}\} \sim \{{\bf v}\}$. Also, if $W$ is a set of words and $\bf u$ is a word then $W \preceq {\bf u}$  means $W \preceq \{{\bf u}\}$.

We use $W^{\le}$ to denote the closure of $W$ under taking subwords and $\langle  W\uparrow \rangle$ to denote the closure of $W$ under going up in order $\preceq$. It is easy to see that $W \subseteq W^{\le} \subseteq \langle  W\uparrow \rangle$ and $W \sim W^{\le} \sim \langle W \uparrow \rangle$. If $W$ is finite then $W^{\le}$ is also finite. On the other hand, if the set $W^{\le}$ contains $ab$ then the set  $\langle  W\uparrow \rangle$ is always infinite, because it contains the words $t_1t_2 \dots t_n$ for arbitrary $n>0$. Proposition \ref{pr1} immediately implies the following.

\begin{prop} \label{pr2} For two sets of words $W$ and $N$ the following
conditions are equivalent:

(i) $W \not \preceq {\bf n}$ for any ${\bf n} \in N$;

(ii) $\operatorname{var} S(W)$ contains none of $S(\{{\bf n}\})$ for any ${\bf n} \in N$;

(iii) ${\bf n} \not \in \langle W \uparrow \rangle$ for any ${\bf n} \in N$.

\end{prop}

We use the word {\em substitution} to refer to the homomorphisms of the free monoid. Since every substitution $\Theta$ is uniquely determined by its values on the letters of the alphabet $\mathfrak A$, we write $\Theta: \mathfrak A \rightarrow \mathfrak A^*$  to remind that the empty word $\epsilon$ is in the range of values.
If $\mathfrak X$ is a set of variables then we write ${\bf u}(\mathfrak X)$ to refer to the word obtained from ${\bf u}$ by deleting all occurrences of all variables that are not in $\mathfrak X$ and say that the word ${\bf u}$ {\em deletes} to the word ${\bf u}(\mathfrak X)$. If $\mathfrak X = \{y_1, \dots, y_k\} \cup \mathfrak Y$ for some variables $y_1, \dots, y_k$ and a set of variables $\mathfrak Y$ then instead of ${\bf u}(\{y_1, \dots, y_k\} \cup \mathfrak Y)$ we simply write ${\bf u}(y_1, \dots, y_k, \mathfrak Y)$.
We say that a set of variables $\mathfrak X$ is {\em stable} in an identity $\bf u \approx \bf v$ if ${\bf u}(\mathfrak X)={\bf v}(\mathfrak X)$. Otherwise, we say that set $\mathfrak X$ is {\em unstable} in $\bf u \approx \bf v$. In particular, a variable $x$ is stable in ${\bf u} \approx {\bf v}$ if and only if it occurs the same number of times in $\bf u$ and $\bf v$.  An identity ${\bf u} \approx {\bf v}$ is called {\em balanced} if every variable is stable in ${\bf u} \approx {\bf v}$.
If a semigroup $S$ satisfies all identities in a set $\Sigma$ then we write $S \models \Sigma$.

\begin{lemma} \label{lemmaOS} \cite[Lemma 2.5]{OS} Let $W$ be a set of words and $\bf u \approx \bf v$ be a balanced identity. Suppose that for every pair of variables $\{x, y\}$ unstable in $\bf u \approx \bf v$ and every substitution $\Theta: \mathfrak A \rightarrow \mathfrak A^*$ such that
 $\Theta(y) \Theta(x) \ne \Theta(x) \Theta(y)$, neither $\Theta({\bf u})$ nor $\Theta({\bf v})$ belongs to $W^{\le}$. Then $S(W) \models \bf u \approx \bf v$.
\end{lemma}

\begin{cor} \label{nfbcombinations} Let $L=L^{\le}$ and $N$ be sets of words and ${\bf u} \approx {\bf v}$ be a balanced identity.
Let $W \subseteq L$ be such that $W \not \preceq {\bf n}$ for any ${\bf n} \in N$.

Suppose that for every pair of variables $\{x, y\}$ unstable in ${\bf u} \approx {\bf v}$ and every substitution $\Theta: \mathfrak A \rightarrow \mathfrak A ^*$  such that $\Theta(x)$ contains some $a \in \mathfrak A$ and $\Theta(y)$ contains $b\ne a$, each of the following conditions is satisfied.

(i) If  $\Theta({\bf u}) \in L$ then $\Theta({\bf u}) \preceq {\bf n}$  for some ${\bf n} \in N$.

(ii) If  $\Theta({\bf v}) \in L$ then $\Theta({\bf v}) \preceq {\bf n}$  for some ${\bf n} \in N$.

Then $S(W) \models {\bf u} \approx {\bf v}$.

\end{cor}

\begin{proof}
Let $\{x, y\}$ be a pair of variables unstable in $\bf u \approx \bf v$ and $\Theta: \mathfrak A \rightarrow \mathfrak A^*$ be a substitution such that
 $\Theta(y) \Theta(x) \ne \Theta(x) \Theta(y)$. Since $\Theta(x) \Theta(y) \ne \Theta(y) \Theta(x)$, without loss of generality, the word
 $\Theta(x)$ contains letter $a$ and the word $\Theta(y)$ contains $b\ne a$.

If  $\Theta({\bf u}) \in W^{\le}$ then by Condition (i) we have that $\Theta({\bf u}) \preceq {\bf n}$  for some ${\bf n} \in  N$.
This contradicts our assumption that $W \not \preceq {\bf n}$ for any ${\bf n} \in N$.
A similar argument shows that $\Theta({\bf v})$ does not belong to $W^{\le}$.
Therefore, Lemma \ref{lemmaOS} implies that $S(W) \models {\bf u} \approx {\bf v}$.
\end{proof}

\section{Syntactic description of the isoterms for certain varieties and hereditary finitely based sets of words}

\begin{fact} \label{xtx} \cite[Fact 3.1]{OS1}
If $xtx$ is an isoterm for a monoid $S$, then

(i) the words $xt_1yxt_2y$ and $xt_1xyt_2y$ can only
form an identity of $S$ with each other;

(ii) the words $xyt_1xt_2y$ and $yxt_1xt_2y$ can only
form an identity of $S$ with each other;

(iii) the words $xt_1yt_2xy$ and $xt_1yt_2yx$ can only
form an identity of $S$ with each other.
\end{fact}

We reserve letter $t$ with or without subscripts to denote linear variables. If we use letter $t$ several times in a word,  we assume that different occurrences of $t$ represent distinct linear variables; so $xtxytyt$ abbreviates $xt_1xyt_2yt_3$ for example. Fact \ref{xtx} immediately implies the following.

\begin{fact} \label{xy3} \cite[Fact 3.2]{OS1} $xtxyty \sim xtyxty$,  $xytxty \sim yxtxty$ and  $xtytxy \sim xtytyx$.

\end{fact}

The  identities $xt_1xyt_2y \approx xt_1yxt_2y$,  $xyt_1xt_2y \approx yxt_1xt_2y$ and  $xt_1yt_2xy \approx xt_1yt_2yx$ we denote respectively by $\sigma_{\mu}$,  $\sigma_1$ and $\sigma_{2}$. Notice that the identities $\sigma_1$ and $\sigma_{2}$ are dual to each other.

If some variable $x$ occurs $n \ge 0$ times in a word ${\bf u}$ then we write $occ_{\bf u}(x)=n$. The set $\con({\bf u}) = \{x \in \mathfrak A \mid occ_{\bf u}(x)>0 \}$ of all variables contained in a word ${\bf u}$ is called the {\em content of ${\bf u}$}.
We use $_{i{\bf u}}x$ to refer to the $i^{th}$ from the left occurrence of variable $x$ in a word ${\bf u}$. We use $_{last{\bf u}}x$ to refer to the last occurrence of $x$ in ${\bf u}$.  The set $\ocs({\bf u}) = \{ {_{i{\bf u}}x} \mid x \in \mathfrak A, 1 \le i \le occ_{\bf u} (x) \}$ of all occurrences of all variables in ${\bf u}$ is called the {\em occurrence set of ${\bf u}$}.
As in \cite{OS2}, with each subset $\Sigma$ of $\{\sigma_1, \sigma_{\mu}, \sigma_2\}$ we associate an assignment of two Types to all pairs of occurrences of distinct non-linear variables in all words as follows. We say that each pair of occurrences of two distinct non-linear variables in each word is $\{\sigma_1, \sigma_{\mu}, \sigma_2\}$-{\em good}. If $\Sigma$ is a proper subset of $\{\sigma_1, \sigma_{\mu}, \sigma_2\}$, then we say that a pair of occurrences of distinct non-linear variables is $\Sigma$-{\em good} if it is  not declared to be $\Sigma$-{\em bad} in the following definition.

\begin{definition} \cite{OS2} \label{goodfact} If $\{c, d\} \subseteq \ocs({\bf u})$ is a pair of occurrences of two distinct non-linear variables $x \ne y$ in a word ${\bf u}$  then

(i)  pair $\{c, d \}$ is  $\{\sigma_{\mu} , \sigma_{2}\}$-bad if $\{c,d\} = \{{_{1{\bf u}}x}, {_{1{\bf u}}y} \}$;

(ii)   pair $\{c, d \}$  is $\{\sigma_1 , \sigma_{\mu}\}$-bad if  $\{c,d\} = \{{_{last{\bf u}}x}, {_{last{\bf u}}y} \}$;

(iii)  pair $\{c, d \}$  is $\{\sigma_1 , \sigma_{2}\}$-bad if   $\{c,d\} = \{{_{1{\bf u}}x}, {_{last{\bf u}}y} \}$.

(iv)  pair $\{c, d \}$  is $\sigma_{\mu}$-bad if $\{c,d\} = \{{_{1{\bf u}}x}, {_{1{\bf u}}y} \}$ or  $\{c,d\} = \{{_{last{\bf u}}x}, {_{last{\bf u}}y} \}$;

(v)  pair $\{c, d \}$  is $\sigma_{2}$-bad  if $c = {_{1{\bf u}}x}$ or $d = {_{1{\bf u}}y}$;

(vi)  pair $\{c, d \}$  is $\sigma_{1}$-bad if  $c = {_{last{\bf u}}x}$ or $d = {_{last{\bf u}}y}$.

\end{definition}

We denote the set of all left sides of identities from $\Sigma$ by $L_{\Sigma}$ and
the set of all right sides of identities from $\Sigma$ by $R_{\Sigma}$. We use $\Sigma^\delta$ to denote the closure of $\Sigma$ under deleting variables.

\begin{lemma} \label{firstsim1} If $S$ is a monoid such that $xtx$ is an isoterm for $S$ and $\Sigma \subseteq \{\sigma_1, \sigma_{\mu}, \sigma_2\}$  then the
following conditions are equivalent:

(i) $S \models \Sigma$;

(ii) if a word $\bf u$ is an isoterm for $S$ then each adjacent pair of occurrences of two distinct non-linear variables in $\bf u$ is $\Sigma$-bad;

(iii) no word in $L_{\Sigma}$ is an isoterm for $S$;

(iv) no word in $R_{\Sigma}$ is an isoterm for $S$.

\end{lemma}

\begin{proof} (i) $\rightarrow$ (ii) Suppose that $\bf u$ contains a $\Sigma$-good adjacent pair $\{c,d \} \subseteq \ocs({\bf u})$ of occurrences of two distinct non-linear variables. Then one of the identities in $\Sigma^\delta$ is applicable to $\bf u$. Therefore,  $S \models {\bf u} \approx {\bf v}$ such that the word $\bf v$ is obtained from $\bf u$ by swapping $c$ and $d$. This contradicts the fact that $\bf u$ is an isoterm for $S$. So, we must assume that every adjacent pair of occurrences of two distinct non-linear variables in $\bf u$ is $\Sigma$-bad.

(ii) $\rightarrow$ (iii) follows from the fact that the only adjacent pair of occurrences of two distinct non-linear variables in each word in $L_\Sigma$ is $\Sigma$-good.

 (iii) $\rightarrow$ (iv) follows immediately from Fact \ref{xy3}.

 (iv) $\rightarrow$ (i) follows immediately from Fact \ref{xtx}.
\end{proof}

 Together with Definition \ref{goodfact}, the following statement gives us explicit syntactic descriptions of the monoids of the form $S(W)$ contained in the seven varieties defined by non-empty subsets of $\{\sigma_1, \sigma_{\mu}, \sigma_2\}$.

\begin{theorem} \label{firstsim} If $W$ is a set of words and $\Sigma \subseteq \{\sigma_1, \sigma_{\mu}, \sigma_2\}$  then the
following conditions are equivalent:

(i) $S(W) \models \Sigma$;

(ii) every adjacent pair of occurrences of two distinct non-linear variables in each word in $W$ is $\Sigma$-bad;

(iii) for each ${\bf u} \in L_{\Sigma}$ we have $W \not \preceq {\bf u}$;

(iv) for each ${\bf u} \in R_{\Sigma}$ we have $W \not \preceq {\bf u}$.

\end{theorem}

\begin{proof} (i) $\rightarrow$ (ii) If $W \not \preceq xtx$ then no two distinct non-linear variables are adjacent in any word in $W$.
If $W \preceq xtx$ then we apply Lemma \ref{firstsim1}.

(ii) $\rightarrow$ (i) Using Lemma \ref{lemmaOS} one can easily show that $S(W) \models \Sigma$.

(i) $\rightarrow$ (iii) Evident.

(iii) $\rightarrow$ (iv)  follows immediately from Fact \ref{xy3}.

(iv) $\rightarrow$ (i) If $W \not \preceq xtx$ then no two distinct non-linear variables are adjacent in any word in $W$.
Consequently, $S(W) \models \{\sigma_1, \sigma_{\mu}, \sigma_2\}$ by Lemma \ref{lemmaOS}.
If $W \preceq xtx$ then follows immediately from Fact \ref{xtx}.
\end{proof}

\begin{lemma} \cite[Theorem 1.1]{EL1} \label{fb4} Every monoid that satisfies the identities $\sigma_1$ and  $\sigma_{\mu}$ is finitely based.
\end{lemma}

 A monoid $S$ is said to be {\em hereditarily finitely based} if every monoid subvariety of $\var S$ is finitely based.
We say that a set of words $W$ is {\em hereditarily finitely based} if every set of words $W'$ with the property $W \preceq W'$ is finitely based.

\begin{cor} \label{11words} For a set of words $W$ the following conditions are equivalent:

(I) the monoid $S(W)$ is hereditarily finitely based;

(II) $W$ is hereditarily finitely based;

(III) $W \not \preceq xtxyty$ and either $W \not \preceq xytxty$ or $W \not \preceq xtytxy$;

(IV) $W$ satisfies one of the following dual conditions:

(i) every adjacent pair of occurrences of two non-linear variables $x \ne y$ in each word ${\bf u} \in W$ is of the form $\{{_{1{\bf u}}x}, {_{1{\bf u}}y} \}$;

(ii) every adjacent pair of occurrences of two non-linear variables $x \ne y$ in each word ${\bf u} \in W$ is of the form $\{{_{last{\bf u}}x}, {_{last{\bf u}}y} \}$.

(V) $S(W)$ is contained either in $\var \{\sigma_1, \sigma_\mu\}$ or in $\var \{\sigma_2, \sigma_\mu\}$.

\end{cor}

\begin{proof} (I) $\rightarrow$ (II) If $S(W)$ is hereditary finitely based then the set of words $W$ is hereditary finitely based by Proposition \ref{pr1}.

(II) $\rightarrow$ (III) Suppose that $W \preceq xtxyty$ or $W \preceq \{xytxty, xtytxy\}$. By the result of Jackson from \cite{MJ}, both $S(\{xtxyty\})$ and  $S(\{xytxty, xtytxy\})$ are NFB.
This contradicts the fact that $W$  hereditarily finitely based.

(III) $\rightarrow$ (IV) Follows from Theorem \ref{firstsim} for $\Sigma = \{\sigma_1, \sigma_{\mu}\}$ and $\Sigma = \{\sigma_2, \sigma_{\mu}\}$.

(IV) $\rightarrow$ (V) Follows from Theorem \ref{firstsim} for $\Sigma = \{\sigma_1, \sigma_{\mu}\}$ and $\Sigma = \{\sigma_2, \sigma_{\mu}\}$.

(V) $\rightarrow$ (I) Follows from the result of Lee \cite{EL1} (see Lemma \ref{fb4} above).
\end{proof}

A {\em block} of a word ${\bf u}$ is a maximal subword of ${\bf u}$ that does not contain any linear variables of ${\bf u}$.
For $n \ge 0$, a word $\bf u$ is called {\em block-$n$-simple} if each block of $\bf u$ involves at most $n$ distinct variables. For example, the word $aabbat_1bcbct_2cca$ is block-2-simple. Evidently, a word $\bf u$ is block-$0$-simple if and only if ${\bf u} = t_1t_2 \dots t_k$ for some $k \ge 0$.
A word that contains at most one non-linear variable is called {\em almost-linear}. We use $\non ({\bf u})$ to denote the set of all non-linear variables in a word $\bf u$.

\begin{cor} \label{blocksim} For a set of words $W$ the following conditions are equivalent:

(i) each word in $W$ is block-1-simple;

(ii) $W$ is equationally equivalent to a set of almost-linear words;

(iii) $S(W)$ is contained in $\var \{\sigma_1, \sigma_\mu, \sigma_2 \}$.

\end{cor}

\begin{proof} (i) $\rightarrow$ (ii) For each ${\bf w} \in W$ and $x \in \non ({\bf w})$ let ${\bf w}_x$ denote the word obtained from $\bf w$ by erasing all occurrences
of all non-linear variables in $\bf w$ except for $x$. Then ${\bf w} \sim \{ {\bf w}_x \mid x \in \non ({\bf w})\}$.
(This easily follows from Lemma 4.1 in \cite{JS} which says that if a word $\bf v$ is obtained by erasing a prefix (suffix) of a block in a word $\bf u$ then  $\{{\bf u}, xy\} \preceq {\bf v}$.)

(ii) $\rightarrow$ (iii) Follows from Theorem \ref{firstsim} for $\Sigma = \{\sigma_1, \sigma_{\mu}, \sigma_2\}$.

(iii) $\rightarrow$ (i) Follows from Theorem \ref{firstsim} for $\Sigma = \{\sigma_1, \sigma_{\mu}, \sigma_2\}$.
\end{proof}

Corollaries \ref{11words} and \ref{blocksim} imply that every set of block-1-simple words is hereditarily finitely based. This generalizes Theorem 3.2 in \cite{OS} which says that every set of almost-linear words is finitely based.

\section{Seven sufficient conditions under which a set of block-2-simple words is NFB}

\begin{table}[tbh]
\begin{center}
\small
\begin{tabular}{|l|l|l|}
\hline set $I$ &identity  ${\bf U}_n \approx {\bf V}_n$ for $n>3$ & set $N$\\
\hline
\protect\rule{0pt}{10pt}$xytxy$ & $[XYn]t[Yn][Xn]
\approx [Yn][Xn]t[XYn]$  &  $xytyx$\\
\hline

\protect  \rule{0pt}{10pt}$xytyx$ & $y[Xn]ty[nX] \approx [Xn]yt[nX]y$ &$xytxy$ \\
\hline
\protect\rule{0pt}{10pt}$xytxty$, $xtytxy$&$[X(n^2)]t[X(n^2)\pi] \approx [X(n^2)\pi]t[X(n^2)]$& $xytxy$, $xytyx$\\
\hline

\protect\rule{0pt}{10pt}$xtxyty$& \parbox{10pt}{$[ZPn]^t \hskip .04in x[ZQn] xy [PRn]y \hskip .04in ^t[QRn] \approx [ZPn]^t \hskip .04in x[ZQn] yx [PRn]y \hskip .04in ^t[QRn]$}& $xxyy$, $xytxy$\\
\hline

\protect\rule{0pt}{10pt}$xxyy$, $xytxty$&  $xytyz_1^2z_2^2 \dots z_n^2x \approx yxtyz_1^2z_2^2 \dots z_n^2x$ & $xytxy$, $xytyx$\\
\hline

\protect\rule{0pt}{10pt}$xtxyty$, $xytxy$, $xytyx$ & $xy[An]yxt[nA]  \approx yx[An]xyt[nA]$&$xyxyx, \{yx^my| m>1\}$ \\
\hline

\protect\rule{0pt}{10pt}$\{ytyx^dtx^{k-d}, x^{k-d}tx^dyty \mid $ & $yt_1x^{k-1}yp_1^2 \dots p_n^2zxt_2z \approx$  &  $x^kyty$, $k>2$ \\

\protect\rule{0pt}{10pt}$ 0 <d <k\}$, $xxyy$ &  $\approx yt_1x^k yp_1^2 \dots p_n^2zt_2z$ &  $xytxty$, $xtytxy$ \\
\hline

\end{tabular}
\caption{Seven NFB intervals $[I, \iso(B2, \Sigma)]$ \protect\rule{0pt}
{11pt}}

\label{classes}
\end{center}
\end{table}

As in \cite{JS}, the words $x_1x_2 \dots x_n$ and $x_nx_{n-1}\dots x_1$ are denoted by $[Xn]$ and $[nX]$ respectively.
The word $x_1 y_1 x_2 y_2 \dots x_{n}y_n$ is denoted by $[XYn]$. We use ${\bf U}^t$ ($^t{\bf U}$) to denote the word obtained from  a word $\bf U$ by inserting a linear variable after (before) each occurrence of each variable in $\bf U$. For example, $[Zn]^t = z_1tz_2t \dots tz_nt$.
The following words for $n>3$ were used by M. Jackson  to prove Lemma 5.4 in \cite{MJ}:
\[[X(n^2)\pi] = (x_1 x_{1 + n} \dots x_{1 + n^2 - n})(x_2 x_{2 + n} \dots x_{2+n^2 -n}) \dots  (x_n x_{2n} \dots x_{n^2}).\]
For each $n>3$, we denote the $n^{th}$ word in this sequence by $[X(n^2)\pi]$ to remind us that it is the result of applying a certain permutation $\pi$ to $[X(n^2)] = x_1 x_2 \dots x_{n^2}$.
We need the following sufficient conditions under which a monoid is non-finitely based.

\begin{lemma} \label{nfbsuf} For every monoid $S$ the following is true:

 (i) \cite[Lemma 4.4]{JS} If the word $xytxy$ is an isoterm for $S$ and for each $n>1$, $S$ satisfies the identity ${\bf U}_n \approx {\bf V}_n$ in Row 1
of Table \ref{classes}, then $S$ is NFB;

(ii) \cite[Lemma 5.2]{OS} If the word $xytyx$ is an isoterm for $S$ and for each $n>1$, $S$ satisfies the identity ${\bf U}_n \approx {\bf V}_n$ in Row 2
of Table \ref{classes}, then $S$ is NFB;

(iii) \cite[Lemma 5.4]{MJ} If the words $xytxy$ and $xytyx$ are isoterms for $S$ and for each $n>3$, $S$ satisfies the identity ${\bf U}_n \approx {\bf V}_n$ in Row 3
of Table \ref{classes}, then $S$ is NFB;

(iv) \cite[Theorem 4.4(iii)]{OS1} If the word $xtxyty$ is an isoterm for $S$ and for each $n>1$, $S$ satisfies the identity ${\bf U}_n \approx {\bf V}_n$ in Row 4 of Table \ref{classes}, then $S$ is NFB;

(v) \cite[Theorem 4.4(iv)]{OS1} If the words $xxyy$ and $xytxty$ are isoterms for $S$ and for each $n>1$, $S$ satisfies the identity ${\bf U}_n \approx {\bf V}_n$ in Row 5 of Table \ref{classes}, then $S$ is NFB;

(vi) \cite[Theorem 4.4(v)]{OS1} If the words $\{xtxyty, xytxy, xytyx\}$ are isoterms for $S$ and for each $n>1$, $S$ satisfies the identity ${\bf U}_n \approx {\bf V}_n$ in Row 6 of Table \ref{classes}, then $S$ is NFB;

(vii) \cite[Theorem 4.4(viii)]{OS1} Fix $k >2$. If the words
\[\{xxyy\} \cup \{ytyx^dtx^{k-d}, x^{k-d}tx^dyty | 0 <d < k\}\]
 are isoterms for $S$ and for each $n>1$, $S$ satisfies the identity ${\bf U}_n \approx {\bf V}_n$ in Row 7 of Table \ref{classes}, then $S$ is NFB.

\end{lemma}

The following two facts can be easily verified and are needed only to prove Theorem \ref{t1}.

\begin{fact} \label{ab} Let $\bf u$ be a word that contains only variables $a$ and $b$.
If $\bf u$ contains an occurrence of $a$ that precedes
an occurrence of $b$ then $\bf u$ contains $ab$ as a subword.
\end{fact}

\begin{fact} \label{abtba} (i) For any set of words $W$ we have $W \preceq xytxy$ if and only if $W^\le$ contains a word of the form $ab{\bf P}ab$ for some possibly empty word $\bf P$ and some distinct  letters $a$ and $b$.

(ii) If each word in $W$ is  block-2-simple then $W \preceq xytyx$ if and only if $W^\le$ contains a word of the form $ab{\bf P} ba$ for some possibly empty word $\bf P$  and some distinct letters $a$ and $b$.
\end{fact}

\begin{theorem} \label{t1} Take sets of words $I$ and $N$ from one of the seven rows in  Table \ref{classes}.
Let $W$ be a set of block-2-simple words such that $W \preceq I$ but $W \not \preceq {\bf n}$ for any ${\bf n} \in N$. Then
$W$ is NFB.
\end{theorem}

\begin{proof} Each time we use Corollary \ref{nfbcombinations} we take $L$ to be the set of all block-2-simple words. Evidently, this set of words is
closed under taking subwords.

 {\bf Row 1 in Table \ref{classes}.}  Here $I = \{xytxy\}$ and $N = \{xytyx\}$.

Each unstable pair of variables  in ${\bf U}_n \approx {\bf V}_n$ is of the form $\{x_i,y_j\}$ for some $1 \le i,j \le n$. If $\{x_i, y_j\}$ is an unstable pair in ${\bf U}_n \approx {\bf V}_n$, then ${\bf U}_n$ deletes to $x_i y_jty_jx_i$.
Let $\Theta: \mathfrak A \rightarrow \mathfrak A ^*$  be a substitution such that $\Theta(x_i)$ contains some letter $a$ and $\Theta(y_j)$ contains $b\ne a$.
If $\Theta({\bf U}_n)$ is a block-2-simple word, then by Fact \ref{ab}, $\Theta([XYn])$ contains $ab$ as a subword.
Similarly, $\Theta([Yn][Xn])$ contains $ba$ as a subword.
Then $\Theta({\bf U}_n)$ contains a subword $ab{\bf P}ba$ for some possibly empty word $\bf P$. Fact \ref{abtba}(ii) implies that $\Theta({\bf U}_n) \preceq xytyx$.  By symmetric arguments, we show that if  $\Theta({\bf V}_n)$ is a block-2-simple word then  $\Theta({\bf V}_n) \preceq xytyx$.

 Corollary \ref{nfbcombinations} implies that for each $n>1$, the monoid $S(W)$ satisfies the identity ${\bf U}_n \approx {\bf V}_n$ in Row 1 of  Table \ref{classes}. The rest follows from Lemma \ref{nfbsuf}(i).

{\bf Row 2 in Table \ref{classes}.}  Here $I = \{xytyx\}$ and $N = \{xytxy\}$.

The only unstable pair of variables  in ${\bf U}_n \approx {\bf V}_n$ are  $\{x_i,y\}$, $i=1, \dots n$. Fix some  $1 \le i \le n$ and a substitution $\Theta: \mathfrak A \rightarrow \mathfrak A ^*$ such that $\Theta(y)$ contains some letter $a$ and $\Theta(x_i)$ contains $b\ne a$.
If  $\Theta({\bf U}_n)$ is a block-2-simple word, by Fact \ref{ab}, the word $\Theta(y[Xn])$ contains $ab$ as a subword. Similarly, $\Theta(y[nX])$ also contains $ab$ as a subword. So, $\Theta({\bf U}_n)$ contains a subword
$ab{\bf P}ab$ for some possibly empty word $\bf P$. Then by Fact \ref{abtba}, we have $\Theta({\bf U}_n) \preceq xytxy$. By symmetric arguments, we show that if  $\Theta({\bf V}_n)$ is a block-2-simple word then  $\Theta({\bf V}_n) \preceq xytxy$.

 Corollary \ref{nfbcombinations} implies that for each $n> 1$, the monoid $S(W)$ satisfies the identity ${\bf U}_n \approx {\bf V}_n$ in Row 2 of  Table \ref{classes}. The rest follows from Lemma \ref{nfbsuf}(ii).

{\bf Row 3 in Table \ref{classes}.}
  Here $I = \{xytxty, xtytxy\}$ and $N = \{xytxy, xytyx\}$.

 Let $\{x_i,x_j\}, 1 \le i<j \le n$ be an unstable pair of variables  in ${\bf U}_n \approx {\bf V}_n$.  Let $\Theta$ be a substitution such that $\Theta(x_i)$ contains some letter $a$ and $\Theta(x_j)$ contains $b\ne a$. If $\Theta({\bf U}_n)$ is a block-2-simple word, by Fact \ref{ab}, the word $\Theta([X(n^2)])$ contains $ab$ as a subword.
Similarly, $\Theta([X(n^2)\pi]))$ contains either $ab$ or $ba$ as a subword. So, $\Theta({\bf U}_n)$ contains a subword $ab{\bf P}ab$ or $ab{\bf P}ba$ for some possibly empty word $\bf P$. Then by Fact \ref{abtba}, we have that either $\Theta({\bf U}_n) \preceq xytxy$ or  $\Theta({\bf U}_n) \preceq xytyx$. By symmetric arguments, we show that if  $\Theta({\bf V}_n)$ is a block-2-simple word then  either $\Theta({\bf V}_n) \preceq xytxy$ or  $\Theta({\bf V}_n) \preceq xytyx$.

 Corollary \ref{nfbcombinations} implies that for each $n> 3$, the monoid $S(W)$ satisfies the identity ${\bf U}_n \approx {\bf V}_n$ in Row 3 of  Table \ref{classes}. The rest follows from Lemma \ref{nfbsuf}(iii).

{\bf Row 4 in Table \ref{classes}.}
Here $I = \{xtxyty\}$ and $N = \{xxyy, xytxy\}$.

The only unstable pair  of variables   in ${\bf U}_n \approx {\bf V}_n$ is $\{x,y \}$.  Let $\Theta$ be a substitution such that $\Theta(x)$ contains some variable $a$ and $\Theta(y)$ contains $b\ne a$.

First we suppose that $\Theta({\bf U}_n)$ is a block-2-simple word.
If $\Theta([ZQn])$ contains $b$ or  $\Theta([PRn])$ contains $a$ then by Fact \ref{ab}, the word $\Theta({\bf U}_n)$ contains a subword $ab{\bf C}ab$ for some possibly empty word $\bf C$. Therefore, $\Theta({\bf U}_n) \preceq xytxy$ by  Fact \ref{abtba}.
So, we can assume that $\Theta([ZQn])=a^k$ for some $k \ge 0$ and $\Theta([PRn])=b^q$ for some $q \ge 0$. If $\Theta(x)$ contains $b$ or $\Theta(y)$ contains $a$, then
in view of Facts \ref{ab} and \ref{abtba}, we have $\Theta({\bf U}_n) \preceq xytxy$. If $\Theta(x)$ is a power of $a$ and $\Theta(y)$ is a power of $b$, then  $\Theta({\bf U}_n) \preceq xxyy$.

  Now we suppose that $\Theta({\bf V}_n)$ is a block-2-simple word.  Then, in view of  Fact \ref{ab}, the word $\Theta(x[ZQn]yx[PRn]y)$ contains  a subword $ab{\bf C}ab$ for some possibly empty word $\bf C$. Therefore, $\Theta({\bf V}_n) \preceq xytxy$ by  Fact \ref{abtba}.

 Corollary \ref{nfbcombinations} implies that for each $n> 1$, the monoid $S(W)$ satisfies the identity ${\bf U}_n \approx {\bf V}_n$ in Row 4 of  Table \ref{classes}. The rest follows from Lemma \ref{nfbsuf}(iv).

{\bf Row 5 in Table \ref{classes}.}
Here $I = \{xxyy, xytxty\}$ and $N = \{xytxy, xytyx\}$.

The only unstable pair  of variables  in ${\bf U}_n \approx {\bf V}_n$ is $\{x,y\}$.  Let $\Theta$ be a substitution such that $\Theta(x)$ contains some letter $a$ and $\Theta(y)$ contains $b\ne a$. If $\Theta({\bf U}_n)$ is a block-2-simple word then  in view of  Fact \ref{ab}, the word $\Theta({\bf U}_n)$ contains  a subword $ab{\bf C}ba$ for some possibly empty word $\bf C$. Therefore, $\Theta({\bf U}_n) \preceq xytyx$ by  Fact \ref{abtba}. If $\Theta({\bf V}_n)$ is a block-2-simple word then by using similar arguments one can show that $\Theta({\bf U}_n) \preceq xytxy$.

 Corollary \ref{nfbcombinations} implies that for each $n> 1$, the monoid $S(W)$ satisfies the identity ${\bf U}_n \approx {\bf V}_n$ in Row 5 of  Table \ref{classes}. The rest follows from Lemma \ref{nfbsuf}(v).

{\bf Row 6 in Table \ref{classes}.}
Here $I = \{xtxyty, xytxy, xytyx\}$ and $N = \{xyxyx\} \cup \{y x^my \mid m>1\}$.

The only unstable pair  of variables   in ${\bf U}_n \approx {\bf V}_n$ is $\{x,y\}$.  Let $\Theta$ be a substitution such that  $\Theta(x)$ contains some letter $a$ and $\Theta(y)$ contains letter $b\ne a$.  If  $\Theta({\bf U}_n)$ is a block-2-simple word, the content of  $\Theta(xy[An]yx)$ is $\{a,b\}$.  Now it is easy to see that modulo renaming letters the word $\Theta(xy[An]yx)$  contains either $ababa$ or $ab^ma$ for some $m>1$ as a subword.  Therefore, $\Theta({\bf U}_n) \preceq xyxyx$ or  $\Theta({\bf U}_n) \preceq yx^my$ for some $m>1$.  If $\Theta({\bf V}_n)$ is a block-2-simple word, then by symmetry  $\Theta({\bf V}_n) \preceq xyxyx$ or  $\Theta({\bf V}_n) \preceq yx^my$ for some $m>1$.

 Corollary \ref{nfbcombinations} implies that for each $n> 1$, the monoid $S(W)$ satisfies the identity ${\bf U}_n \approx {\bf V}_n$ in Row 6 of  Table \ref{classes}. The rest follows from Lemma \ref{nfbsuf}(vi).

{\bf Row 7 in Table \ref{classes}.} Fix $k>2$.
Here $I = \{xxyy\} \cup \{ytyx^dtx^{k-d}, x^{k-d}tx^dyty | 0 <d <k\}$ and $N = \{x^kyty, xytxty, xtytxy\}$.

Each unstable pair of variables  in ${\bf U}_n \approx {\bf V}_n$ is of the form
$\{x,y\}$ or $\{x,z\}$ or $\{x,p_i\}$ for some $1 \le i \le n$.

Let $\Theta$ be a substitution such that $\Theta(x)$ contains some $a \in \mathfrak A$. If $\Theta (x)$ is not a power of $a$ then
$\Theta(x^{k-1}) \preceq xytxty \sim xytytx$ and consequently, $\Theta({\bf U}_n)\preceq xytxty \in N$ and  $\Theta({\bf V}_n)\preceq xytxty \in N$.
So, we can assume that $\Theta(x) = a^p$ for some $p>0$. Consider three cases.

{\bf Case 1:} $\Theta(y)$ contains $b\ne a$.
If $\Theta({\bf U}_n)$ is a block-2-simple word, then by Fact \ref{ab}, $\Theta(x^{k-1}yp_1^2 \dots p_n^2z)$
contains $a^{k-1}b$ as a subword. Then $\Theta({\bf U}_n) \preceq btatab \in N$ and $\Theta({\bf V}_n) \preceq btatab \sim xtytxy \in N$.

{\bf Case 2:} $\Theta(z)$ contains $b\ne a$.

If $\Theta({\bf U}_n)$ is a block-2-simple word, then by Fact \ref{ab}, $\Theta(x^{k-1}yp_1^2 \dots p_n^2zx)$
contains $ab{\bf C}a$ as a subword for some word ${\bf C} \in\{a,b\}^*$. Then  $\Theta({\bf U}_n) \preceq abtatb \in N$.

If $\Theta({\bf V}_n)$ is a block-2-simple word, then by Fact \ref{ab}, $\Theta(x^{k}yp_1^2 \dots p_n^2z)$
contains $a^{k}b$ as a subword. Then  $\Theta({\bf V}_n) \preceq a^kbtb \in N$.

{\bf Case 3:} For some $1\le i \le n$, $\Theta(p_i)$ contains $b\ne a$.

If $\Theta({\bf U}_n)$ is a block-2-simple word, then by Fact \ref{ab}, $\Theta(x^{k-1}yp_1^2 \dots p_n^2zx)$
contains $ab{\bf C}ba$ as a subword for some word ${\bf C} \in\{a,b\}^*$. Then  $\Theta({\bf U}_n) \preceq abtba \preceq xytxty \in N$.

If $\Theta({\bf V}_n)$ is a block-2-simple word, then by Fact \ref{ab}, $\Theta(x^{k}yp_1^2 \dots p_n^2z)$
contains $a^kb{\bf C}b$ as a subword for some word ${\bf C} \in\{a,b\}^*$. Then  $\Theta({\bf V}_n) \preceq a^kbtb \in N$.

 Corollary \ref{nfbcombinations} implies that for each $k>2$ and $n> 1$, the monoid $S(W)$ satisfies the identity ${\bf U}_n \approx {\bf V}_n$ in Row 7 of  Table \ref{classes}. The rest follows from Lemma \ref{nfbsuf}(vii).
\end{proof}

Theorem \ref{t1}(i)--(iii) immediately imply the following.

\begin{cor} \label{fb2} Let $W$ be a set of block-2-simple words such that $W\preceq \{xytxty, xtytxy\}$.
Then either $W$ is NFB or $W \preceq \{xytxy, xytyx\}$.
\end{cor}

Corollary \ref{fb2} and Theorem \ref{t1}(vi) immediately implies the following.

\begin{cor} \label{fb3}  Let $W$ be a set of block-2-simple words such that $W\preceq \{xytxty, xtytxy, xtxyty\}$.
Then either $W$ is NFB or $W \preceq xyxyx$ or $W \preceq xy^mx$ for some $m>1$.
\end{cor}

\begin{cor} \label{fb1l} Let $W$ be a set of block-2-simple words such that $W \preceq xtxyty$ but one of the words $\{xytxty, xtytxy\}$ is not an isoterm for $S(W)$. Then either $W$ is NFB or $W \preceq xxyy$ and both words $\{xytxty, xtytxy\}$ are not isoterms for $S(W)$.
\end{cor}

\begin{proof} Notice that $xytxy \preceq  \{xytxty, xtytxy\}$ and  $xytyx \preceq  \{xytxty, xtytxy\}$.  Since
one of the words $\{xytxty, xtytxy\}$ is not an isoterm for $S(W)$, neither $xytxy$ nor $xytyx$ is an isoterm for $S(W)$.
Theorem \ref{t1}(iv) implies that either $W$ is NFB or the word $xxyy$ is an isoterm for $S(W)$. Now
Theorem \ref{t1}(v) implies that either $W$ is NFB or the word $xytxty$ is not an isoterm for $S(W)$.
The dual argument shows that either $W$ is NFB or the word $xtytxy$ is not an isoterm for $S(W)$.
\end{proof}

\section{Words with two non-linear variables and long blocks are NFB }

Let $\bf u$ be a word containing the variables $a$ and $b$. Following Definition 2.4 in \cite{MJ1}, we use $\tilde{\bf u}$ to denote the word obtained from $\bf u$ by replacing all maximal subwords of $\bf u$ not containing the variables $a$ or $b$  by linear variables and by replacing all subwords of the form $ab$ by words of the form $atb$, where $t$ is a linear variable.

\begin{lemma} \cite[Theorem 2.7] {MJ1} \label{longNFB} Let ${\bf w}={\bf w}_1a^{\alpha_1}b^{\beta_1}{\bf w}_2a^{\alpha_2}{\bf p}b^{\beta_2}{\bf w}_3$
be a word such that $a$ and $b$ are variables, $\bf p$, ${\bf w}_1$, ${\bf w}_2$ and ${\bf w}_3$ are possibly empty words and $\alpha_1, \alpha_2, \beta_1, \beta_2 >0$
are maximal. If both $\bf w$ and $xytyx$ are isoterms for a monoid $S$ and for each $n>0$ the word
${\bf w}=\tilde {\bf w}_1a^{\alpha_1}[Xn]b^{\beta_1-1}\tilde {\bf w}_2a^{\alpha_2}t[nX]tb^{\beta_2}\tilde {\bf w}_3$
is not an isoterm for $S$, then $S$ is NFB.
\end{lemma}

[Theorem 2.7] in \cite{MJ1} is a modified and generalized version of  Lemma 5.3 in \cite{OS} that we used to show that a long word in two variables is NFB. Now we are going to use  \cite[Theorem 2.7]{MJ1}  to show that a word with a long block in two variables is NFB.

\begin{lemma} \label{abba} Let $\bf U$ be a word such that all variables in $\bf U$ other than $a$ and $b$ are linear. If $\bf U$ contains a subword $a b^{\beta}a$  for  some $ \beta > 1$ then $\bf U$ is NFB.
\end{lemma}

\begin{proof} If $\bf U$ does not contain any  occurrence of $b$  outside of the word  $a b^{\beta}a$ then $\bf U$ is NFB by Corollary \ref{fb2}.
 So, without  loss of generality we assume that  ${\bf U}={\bf u}_1a^{\alpha_1} b^{\beta}a^{\alpha_2}{\bf v}b^{\gamma}{\bf u}_2$ or ${\bf U}={\bf u}_1a^{\alpha_1} b^{\beta}a^{\alpha_2}{\bf w}a^{\alpha_3}{\bf v}b^{\gamma}{\bf u}_2$.
In both cases, $\alpha_1, \alpha_2, \alpha_3, \gamma >0$, $\alpha_1$ and $\gamma$ are maximal, ${\bf u}_1$, $\bf v$ and ${\bf u}_2$ are possibly empty words such that  if $\bf v$ is not empty then $\bf v$ contains only linear variables. The word $\bf w$ starts and begins with a linear variable and does contain any occurrences of $b$.
Each possibility can be handled by using Lemma \ref{longNFB} in a similar way.

If  ${\bf U}={\bf u}_1a^{\alpha_1} b^{\beta}a^{\alpha_2}{\bf v}b^{\gamma}{\bf u}_2$, then
we use Lemma \ref{longNFB} for ${\bf w}_1={\bf u}_1$, ${\bf w}_2=1$, $\bf p=v$ and  ${\bf w}_3={\bf u}_2$ and show that for each $n>0$ the monoid $S(\{\bf U\})$ satisfies the following identity:
\[{\bf u}_n= \tilde {\bf u}_1 x^{\alpha_1}[An]y^{\beta-1}x^{\alpha_2}t_1[nA]t_2y^{\gamma}  \tilde {\bf u}_2 \approx  \tilde {\bf u}_1 x^{\alpha_1}[An]x^{\alpha_2}y^{\beta-1}t_1[nA]t_2y^{\gamma}  \tilde {\bf u}_2
= {\bf v}_n,\] where ${\bf u}_1$ and ${\bf u}_2$ are written in $x$ and $y$ instead of $a$ and $b$.

If   ${\bf U}={\bf u}_1a^{\alpha_1} b^{\beta}a^{\alpha_2}wa^{\alpha_3}{\bf v}b^{\gamma}{\bf u}_2$, then we use Lemma \ref{longNFB} for ${\bf w}_1={\bf u}_1$, ${\bf w}_2=a^{\alpha_2}\bf w$, $\bf p=v$ and  ${\bf w}_3={\bf u}_2$ and show that for each $n>0$ the monoid $S(\{\bf U\})$ satisfies the following identity:
\[{\bf u}_n= \tilde {\bf u}_1 x^{\alpha_1}[An]y^{\beta-1}x^{\alpha_2}\tilde {\bf w}x^{\alpha_3}t_1[nA]t_2y^{\gamma}  \tilde {\bf u}_2 \approx
 \tilde {\bf u}_1 x^{\alpha_1}[An]x^{\alpha_2}y^{\beta-1} \tilde {\bf w}x^{\alpha_3}t_1[nA]t_2y^{\gamma}  \tilde {\bf u}_2 = {\bf v}_n,\]
 where ${\bf u}_1$ and ${\bf u}_2$ are written in $x$ and $y$ instead of $a$ and $b$.

 Notice that for each $n>0$, $\{x,y\}$ is the only unstable pair of variables in ${\bf u}_n \approx {\bf v}_n$.
 Let $\Theta: \mathfrak A \rightarrow \mathfrak A^*$ be a substitution such that $\Theta(x) \Theta(y) \ne \Theta(y) \Theta(x)$. Then $\Theta(x)$ contains, say, $a$ and $\Theta(y)$ contains $b$ or vice versa.

 Let $m$ denote the total number of occurrences of non-linear variables ($a$ and $b$) in $\bf U$. Notice that $x$ occurs in  ${\bf u}_n$ and ${\bf v}_n$ the same number of times as $a$ in $\bf U$ and the number of occurrences of $y$
 in  ${\bf u}_n$ and ${\bf v}_n$ is one less than the number of occurrences of $b$ in $\bf U$.
If $\Theta([An]) \ne \epsilon$ then  $\Theta({\bf u}_n)$ ($\Theta({\bf v}_n)$, resp.) contains at least $m+1$ occurrences of non-linear letters. Therefore, we can assume that $\Theta([An]) = \Theta([nA])= \epsilon$.

If $\Theta(x)$ contains $a$ then  $\Theta(x)=a$ and  $\Theta(y)=b$.  But then neither $\Theta({\bf u}_n)$ nor $\Theta({\bf v}_n)$  has the word $b^{\beta}$ between $a^{\alpha_1}$
and $a^{\alpha_2}$.

If $\Theta(x)$ contains $b$ then  $\Theta(x)=b$ and  $\Theta(y)=a$. In this case $occ_{\bf U}(a)=occ_{{\bf u}_n}(x)=occ_{{\bf v}_n}(x) \le occ_{\bf U}(b)$ and  $occ_{\bf U}(b)=occ_{{\bf u}_n}(y) +1=occ_{{\bf v}_n}(y) +1 \le occ_{\bf U}(a)+1$. So, either  $occ_{\bf U}(b)=occ_{\bf U}(a)$ or  $occ_{\bf U}(b)=occ_{\bf U}(a)+1$.

 If $occ_{\bf U}(b)=occ_{\bf U}(a)=occ_{{\bf u}_n}(x)=occ_{{\bf v}_n}(x)= occ_{{\bf u}_n}(y) +1=occ_{{\bf v}_n}(y) +1$, then:

(i) The image of no variable other than $x$ contains $b$;

(ii) There is at most one variable $t\ne y$ whose image contains $a$.
If $\Theta(t)$ contains $a$ for some $t \ne y$
then $t$ is linear in ${\bf u}_n$ (in ${\bf v}_n$, resp.) and the variable $a$ occurs only once in $\Theta(t)$.

 If $occ_{\bf U}(b)=occ_{\bf U}(a) +1=occ_{{\bf u}_n}(x) +1 =occ_{{\bf v}_n}(x) +1= occ_{{\bf u}_n}(y) +1=occ_{{\bf v}_n}(y) +1$, then:

(i) The image of no variable other than $y$ contains $a$;

(ii) There is at most one variable $t \ne x$ whose image contains $b$. If $\Theta(t)$ contains $b$ for some $t \ne y$
then $t$ is linear in ${\bf u}_n$ (in ${\bf v}_n$, resp.) and the variable $b$ occurs only once in $\Theta(t)$.

If  $\Theta({\bf u}_n)$ ($\Theta({\bf v}_n)$, resp.) is a subword of $\bf U$ then in view of Conditions (i)-(ii) we have that  $\Theta({\bf u}_n)(a,b)$ ($\Theta({\bf v}_n)(a,b)$, resp.) is a prefix or suffix of ${\bf U}(a,b)$. Since  $\Theta({\bf u}_n)(a,b)$ and ${\bf U}(a,b)$ start and end with different letters, the word  $\Theta({\bf u}_n)(a,b)$ can be neither prefix nor suffix of  ${\bf U}(a,b)$.  Since  $\Theta({\bf v}_n)(a,b)$ and ${\bf U}(a,b)$ start  and begin with different letters, the word  $\Theta({\bf v}_n)(a,b)$ can  be neither prefix nor suffix of  ${\bf U}(a,b)$.

Overall, neither  $\Theta({\bf u}_n)$  nor $\Theta({\bf v}_n)$ is a subword of $\bf U$. By Lemma \ref{lemmaOS}, the monoid $S(\{\bf U\})$ satisfies the identity ${\bf u}_n \approx {\bf v}_n$ for each $n>0$. Therefore, $\bf U$ is NFB by  Lemma \ref{longNFB}.
\end{proof}

\begin{lemma} \label{ababa}  Let $\bf U$ be a word such that all variables in $\bf U$ other than $a$ and $b$ are linear.  If $\bf U$ contains a subword $ababa$  then $\bf U$ is NFB.
\end{lemma}

\begin{proof} We have that ${\bf U}= {\bf u}_1a^pba ba^q {\bf u}_2$  for some possibly empty words ${\bf u}_1$ and ${\bf u}_2$ so that $p,q>0$ are maximal.
We use Lemma \ref{longNFB} for ${\bf w}_1={\bf u}_1$, ${\bf w}_2={\bf p}=\epsilon$ and  ${\bf w}_3=a^q{\bf u}_2$.

 Let us check that for each $n>0$ the monoid $S(\{\bf U\})$ satisfies the following identity:
\[{\bf u}_n= \tilde {\bf u}_1 x^p[An]xt_1[nA]t_2yx^q  \tilde {\bf u}_2 \approx  \tilde {\bf u}_1 x^p[An]xt_1[nA]t_2x^qy  \tilde {\bf u}_2 = {\bf v}_n,\]
 where ${\bf u}_1$ and ${\bf u}_2$ are written in $x$ and $y$ instead of $a$ and $b$.

 Notice that for each $n>0$, $\{x,y\}$ is the only unstable pair of variables in ${\bf u}_n \approx {\bf v}_n$.
 Let $\Theta: \mathfrak A \rightarrow \mathfrak A^*$ be a substitution such that $\Theta(x) \Theta(y) \ne \Theta(y) \Theta(x)$. Then $\Theta(x)$ contains, say, $a$ and $\Theta(y)$ contains $b$ or vice versa.

{\bf Case 1}: Variable $b$ occurs twice in $\bf U$.

Notice that variable $a$ occurs $m \ge 3$ times in $\bf U$. Since variable $x$ occurs $m$ times  in  ${\bf u}_n$ and ${\bf v}_n$, we have that $\Theta(x)=a$ and $occ_{\bf U}(a)=occ_{\Theta({\bf u}_n)}(a)=occ_{\Theta({\bf v}_n)}(a)$.
If $\Theta([An]) \ne \epsilon$ then  $\Theta([An])$ and $\Theta([nA])$ must contain $b$ and $\Theta({\bf u}_n)$ ($\Theta({\bf v}_n)$, resp.) would contain at least three $b$-s. Therefore, $\Theta([An]) = \Theta([nA])= \epsilon$. But then neither $\Theta({\bf u}_n)$ nor $\Theta({\bf v}_n)$  has a $b$ between the displayed occurrences of $a^p$ and $a$. Therefore, neither $\Theta({\bf u}_n)$ nor $\Theta({\bf v}_n)$ is a subword of $\bf U$.

{\bf Case 2}: Variable $b$ occurs at least three times in $\bf U$.

 Let $m$ denote the total number of occurrences of non-linear variables ($a$ and $b$) in $\bf U$. Notice that $x$ occurs in  ${\bf u}_n$ and ${\bf v}_n$ the same number of times as $a$ in $\bf U$ and the number of occurrences of $y$ in  ${\bf u}_n$ and ${\bf v}_n$ is one less than the number of occurrences of $b$ in $\bf U$.
If $\Theta([An]) \ne \epsilon$ then  $\Theta({\bf u}_n)$ ($\Theta({\bf v}_n)$, resp.) contains at least $m+1$ occurrences of non-linear variables. Therefore, we can assume that $\Theta([An]) = \Theta([nA])= \epsilon$.

If $\Theta(x)$ contains $a$ then  $\Theta(x)=a$ and $occ_{\bf U}(a)=occ_{\Theta({\bf u}_n)}(a)=occ_{\Theta({\bf v}_n)}(a)$. In this case, neither $\Theta({\bf u}_n)$ nor $\Theta({\bf v}_n)$  has a $b$ between the displayed occurrences of $a^p$ and $a$. Therefore, neither $\Theta({\bf u}_n)$ nor $\Theta({\bf v}_n)$ is a subword of $\bf U$.

If $\Theta(x)$ contains $b$ then  $\Theta(x)=b$ and  $\Theta(y)=a$. In this case $occ_{\bf U}(a)=occ_{{\bf u}_n}(x)=occ_{{\bf v}_n}(x) \le occ_{\bf U}(b)$ and  $occ_{\bf U}(b)=occ_{{\bf u}_n}(y) +1=occ_{{\bf v}_n}(y) +1 \le occ_{\bf U}(a)+1$. So, either  $occ_{\bf U}(b)=occ_{\bf U}(a)$ or  $occ_{\bf U}(b)=occ_{\bf U}(a)+1$.

 If $occ_{\bf U}(b)=occ_{\bf U}(a)=occ_{{\bf u}_n}(x)=occ_{{\bf v}_n}(x)= occ_{{\bf u}_n}(y) +1=occ_{{\bf v}_n}(y) +1$, then:

(i) The image of no variable other than $x$ contains $b$;

(ii) There is at most one variable $t\ne y$ whose image contains $a$.
If $\Theta(t)$ contains $a$ for some $t \ne y$
then $t$ is linear in ${\bf u}_n$ (in ${\bf v}_n$, resp.) and the variable $a$ occurs only once in $\Theta(t)$.

 If $occ_{\bf U}(b)=occ_{\bf U}(a) +1=occ_{{\bf u}_n}(x) +1 =occ_{{\bf v}_n}(x) +1= occ_{{\bf u}_n}(y) +1=occ_{{\bf v}_n}(y) +1$, then:

(i) The image of no other variable than $y$ contains $a$;

(ii) There is at most one variable $t \ne x$ whose image contains $b$. If $\Theta(t)$ contains $b$ for some $t \ne y$
then $t$ is linear in ${\bf u}_n$ (in ${\bf v}_n$, resp.) and the variable $b$ occurs only once in $\Theta(t)$.

If  $\Theta({\bf u}_n)$ ($\Theta({\bf v}_n)$, resp.) is a subword of $\bf U$ then in view of Conditions (i)-(ii) we have that  $\Theta({\bf u}_n)(a,b)$ ($\Theta({\bf v}_n)(a,b)$, resp) is a prefix or a suffix of ${\bf U}(a,b)$. Since  $\Theta({\bf u}_n)(a,b)$ and ${\bf U}(a,b)$ start and end with different variables, the word  $\Theta({\bf u}_n)(a,b)$ is neither a prefix nor a suffix of  ${\bf U}(a,b)$.  Since  $\Theta({\bf v}_n)(a,b)$ and ${\bf U}(a,b)$ start with different variables, the word  $\Theta({\bf u}_n)(a,b)$ is not a prefix of  ${\bf U}(a,b)$. If the word ${\bf u}_2$ contains some non-linear variables, then the words
 $\Theta({\bf v}_n)(a,b)$ and ${\bf U}(a,b)$ end with different variables and consequently,  $\Theta({\bf v}_n)(a,b)$ is not a suffix of  ${\bf U}(a,b)$.
If the word ${\bf u}_2$ does not contain any linear variables, then the word  $\Theta({\bf v}_n)(a,b)$ ends with $b^{p+1+q}a$ but the word
  ${\bf U}(a,b)$  ends with  $aba^q$ and consequently, the word  $\Theta({\bf v}_n)(a,b)$ is not a suffix of ${\bf U}(a,b)$.

Overall, neither  $\Theta({\bf u}_n)$  nor $\Theta({\bf v}_n)$ is a subword of $\bf U$. By Lemma \ref{lemmaOS}, the monoid $S(\{\bf U\})$ satisfies the identity ${\bf u}_n \approx {\bf v}_n$ for each $n>0$. Therefore, $\bf U$ is NFB by  Lemma \ref{longNFB}.
\end{proof}

\section{Some finitely based words which are not hereditary finitely based}

\begin{lemma} \label{GL} \cite[Corollary 5.3, Sect.11]{GL}  Two words in a free monoid commute if and only if they are powers of the same word.
\end{lemma}

 The following theorem can be proved by using induction on the maximal length of $\Theta(x)$ and $\Theta(y)$ and Lemma 5.1 in Sect.11 in \cite{GL}. Victor Guba noticed that the word mentioned in Theorem \ref{com} is a generator of the cyclic group $G$ generated by $\Theta(x)$ and $\Theta(y)$ and
 suggested to prove this theorem in a similar way as Lyndon and Schupp proved Proposition 2.17 in \cite{LS}. This computation-free proof is the one that we present.

\begin{theorem} \label{com} Let ${\bf u} \approx {\bf v}$ be a non-trivial identity in two variables $x$ and $y$ and $\Theta: \mathfrak A \rightarrow \mathfrak A ^*$ be a substitution.  If $\Theta({\bf u})=\Theta({\bf v})$ then both $\Theta(x)$ and $\Theta(y)$ are powers of the same word.
\end{theorem}

\begin{proof} Consider the free group $F = \langle \{x_1, \dots, x_n, x'_1, \dots, x'_n \} \mid x_ix_i' = x_i'x_i = \epsilon, 1 \le i \le n \rangle$ where
$\{x_1, \dots, x_n\} = \con(\Theta(xy))$.
By the Nielsen-Schreier Theorem \cite{Sch}, the subgroup $G$ of $F$ generated by $\Theta(x)$ and $\Theta(y)$ is itself free.
By Proposition 2.7 in \cite{LS}, the group $G$ is of rank at most two and if G has rank two, then $\Theta(x)$ and $\Theta(y)$ must form a basis for $G$. But in this case $\Theta({\bf u}) \ne \Theta({\bf v})$. So, the group $G$ is cyclic and consequently, $\Theta(x)$ and $\Theta(y)$ commute. Therefore, by Lemma \ref{GL} they are powers of the same word.
\end{proof}

\begin{cor} \label{subSW2} Let $W$ be a set of words and ${\bf u}_2(x,y)\approx {\bf v}_2(x,y)$ be a non-trivial balanced identity in two variables $x$ and $y$.
Suppose also that for some possibly empty balanced identities ${\bf u}_1\approx {\bf v}_1$ and ${\bf u}_3\approx {\bf v}_3$, $\{x,y\}$ is the only unstable pair of variables in ${\bf u} = {\bf u}_1 {\bf u}_2(x,y){\bf u}_3 \approx {\bf u}_1{\bf v}_2(x,y){\bf v}_3 = {\bf v}$. Then $S(W) \models {\bf u} \approx {\bf v}$ if and only if for every substitution $\Theta: \mathfrak A \rightarrow \mathfrak A^*$ such that $\Theta(y) \Theta(x) \ne \Theta(x) \Theta(y)$, neither $\Theta({\bf u})$ nor $\Theta({\bf v})$ belongs to $W^{\le}$.
\end{cor}

 \begin{proof} The `if' part follows immediately from Lemma \ref{lemmaOS}.

Suppose that $S(W) \models {\bf u} \approx \bf v$. Let $\Theta: \mathfrak A \rightarrow \mathfrak A^*$ be a substitution
 such that $\Theta({\bf u}) \in W^{\le}$.
 Then $\Theta({\bf u}) =\Theta({\bf v})$, and consequently, $\Theta({\bf u}_2(x,y)) = \Theta({\bf v}_2(x,y))$. Theorem \ref{com} implies that $\Theta(y) \Theta(x) = \Theta(x) \Theta(y)$.
\end{proof}

We use $\lin({\bf u})$ to denote the set of all linear variables in a word $\bf u$.
An identity ${\bf u} \approx {\bf v}$ is called {\em block-balanced} if for each variable $x \in \mathfrak A$, we have ${\bf u}(x, \lin({\bf u})) = {\bf v}(x, \lin({\bf u}))$. Evidently, an identity ${\bf u} \approx {\bf v}$ is block-balanced if and only if it is balanced, the order of linear variables is the same in $\bf u$ and $\bf v$ and each block in $\bf u$ is a permutation of the corresponding block in $\bf v$. Corollary \ref{subSW2} immediately imply the following.

\begin{cor} \label{subSW1} Let $W$ be a set of words and $\bf u \approx \bf v$ be a non-trivial block-balanced identity with exactly two non-linear variables $x \ne y$. Then $S(W) \models \bf u \approx \bf v$ if and only if for every substitution $\Theta: \mathfrak A \rightarrow \mathfrak A^*$ such that $\Theta(y) \Theta(x) \ne \Theta(x) \Theta(y)$, neither $\Theta({\bf u})$ nor $\Theta({\bf v})$ belongs to $W^{\le}$.
\end{cor}

 We say that a pair of variables $\{x,y\}$ is {\em b-unstable} in a word $\bf u$ with respect to a semigroup $S$ if $S$ satisfies a block-balanced identity of the form $\bf u \approx \bf v$ such that ${\bf u}(x,y) \ne {\bf v}(x,y)$. Otherwise, we say that $\{x,y\}$ is {\em b-stable} in $\bf u$ with respect to $S$.
Corollary \ref{subSW1} immediately imply the following.

\begin{cor} \label{subSW} Let $W$ be a set of words and $\bf u$ be a word with exactly two non-linear variables $x$ and $y$.
Suppose that one can find a substitution $\Theta: \mathfrak A \rightarrow \mathfrak A^*$ such that $\Theta(y) \Theta(x) \ne \Theta(x) \Theta(y)$ and
$\Theta({\bf u}) \in W^{\le}$. Then the pair $\{x,y\}$ is b-stable in $\bf u$ with respect to $S(W)$.
\end{cor}

\begin{lemma} \label{fbtlem} \cite[Theorem 4.12]{OS2} Let $S$ be a monoid such that for some $m > 1$, the word $x^my^m$ is an isoterm for $S$ and
$S \models \{\sigma_1, \sigma_{2}, x^{m}txyty \approx  x^{m}tyxty\}$. Suppose also that for each $1< k \le m$, $S$ satisfies each of the following dual conditions:

(i) If for some almost-linear word ${\bf A}x$ with $occ_{\bf A}(x)>0$ the pair $\{x,y\}$ is b-unstable in ${\bf A}xy^k$ with respect to $S$ then $S$ satisfies the identity ${\bf A}xy^{k-1}ty \approx {\bf A}yxy^{k-2}ty$;

(ii) If for some almost-linear word $y{\bf B}$  with $occ_{\bf B}(y)>0$ the pair $\{x,y\}$ is b-unstable in $x^ky{\bf B}$ with respect to $S$ then $S$ satisfies the identity $xtx^{k-1}y{\bf B} \approx  xtx^{k-2}yx{\bf B}$.

 Then $S$ is finitely based.
\end{lemma}

\begin{lemma} \label{oneword} Let $\bf U$ be a word with exactly two non-linear variables $a$ and $b$ such that ${\bf U} \not \preceq xtytxy$, ${\bf U} \not \preceq xytxty$ and $m > 1$ be the maximum such that ${\bf U} \preceq x^my^m$. Then

(i) modulo renaming variables, ${\bf U}={\bf C}t_1a^{\alpha}b^{\beta}t_2{\bf B}$ for some possible empty almost-linear words ${\bf C}t_1={\bf C}(a, \lin({\bf C}))t_1$ and $t_2{\bf B} = t_2{\bf B}(b, \lin({\bf B}))$ such that $min(\alpha, \beta) = m$;

(ii) $S(\{{\bf U}\})$ satisfies Conditions (i) and (ii) in Lemma \ref{fbtlem}.

\end{lemma}

\begin{proof} (i) Using Theorem \ref{firstsim} for $\Sigma = \{\sigma_1, \sigma_2\}$ we conclude that every adjacent pair of occurrences $a$ and $b$ in ${\bf U}$ is of the form $\{{_{1{\bf u}}a}, {_{last{\bf u}}b} \}$ or $\{{_{1{\bf u}}b}, {_{last{\bf u}}a}\}$. Therefore, the word $\bf U$ contains only one adjacent
 pair of occurrences $a$ and $b$. Since $m > 1$ the word $\bf U$ must be as described.

(ii) Since Conditions (i) and (ii) in Lemma \ref{fbtlem} are dual, we check only Condition (i).
Let $1<k \le m$ and ${\bf A}x$ be an almost-linear word with $occ_{\bf A}(x)>0$ such that the pair $\{x,y\}$ is b-unstable in ${\bf A}xy^k$ with respect to $S(\{\bf U \})$. Let $\Theta: \mathfrak A \rightarrow \mathfrak A^*$ be a substitution such that $\Theta(y) \Theta(x) \ne \Theta(x) \Theta(y)$. Evidently, $\Theta({\bf A}yxy^{k-2}ty)$ is not a subword of $\bf U$. If $\Theta({\bf A}xy^{k-1}ty)$ is a subword of ${\bf U}={\bf C}t_1a^{\alpha}b^{\beta}t_2{\bf B}$, then $\Theta(x)=a^l$, $\Theta(y)=b^r$ for some $l,r>0$ and the word $\Theta({\bf A}x)$ is a suffix of ${\bf C}t_1a^{\alpha}$. Let $\Theta'$ be a substitution which coincides with $\Theta$ on all variables other than $y$ and
$\Theta'(y)=b$. Since $k \le \beta$, the word $\Theta'({\bf A}xy^k)$ is a subword of $\bf U$. Then by Corollary \ref{subSW}, the pair $\{x,y\}$ is b-stable in ${\bf A}xy^k$ with respect to $S(\{\bf U \})$. To avoid a contradiction, we conclude that the word $\Theta({\bf A}xy^{k-1}ty)$ is not a subword of $\bf U$. So, by Lemma \ref{lemmaOS}, we have that $S(\{{\bf U}\}) \models {\bf A}xy^{k-1}ty \approx {\bf A}yxy^{k-2}ty$. This means that $S(\{{\bf U}\})$ satisfies Condition (i) in Lemma \ref{fbtlem}.
\end{proof}

\begin{lemma} \label{ident} Let $\bf U$ be a word with two non-linear variables such that ${\bf U} \not \preceq xtytxy$, ${\bf U} \not \preceq xytxty$  and
$m > 1$ be the maximum such that ${\bf U} \preceq x^my^m$. Then the following conditions are equivalent:

(i)  ${\bf U} \not \preceq \{ x^{m}txyty,  ytyxtx^{m}\}$;

(ii)  One of the words $\{a^m{\bf U}_1ab{\bf U}_2b, a{\bf U}_1ab{\bf U}_2b^m \}$ is not a subword of $\bf U$ for any ${\bf U}_1, {\bf U}_2 \in \mathfrak A ^*$;

(iii) The monoid $S(\{{\bf U}\})$ satisfies either $x^{m}txyty \approx x^{m}tyxty$ or $xtxyty^{m} \approx xtyxty^{m}$.

\end{lemma}

\begin{proof} In view of Lemma \ref{oneword}, modulo renaming variables, ${\bf U}={\bf C}t_1a^{\alpha}b^{\beta}t_2{\bf B}$ for some possible empty almost-linear words ${\bf C}t_1={\bf C}(a, \lin({\bf C}))t_1$ and $t_2{\bf B} = t_2{\bf B}(b, \lin({\bf B}))$ such that $min(\alpha, \beta) \ge m>1$.

$\neg (ii) \rightarrow \neg (i)$ Suppose that for some ${\bf U}_1, {\bf U}_2, {\bf U}_3, {\bf U}_4  \in \mathfrak A ^*$, the word $\bf U$ contains
a subword   $a^m{\bf U}_1ab{\bf U}_2b$ and a subword $a{\bf U}_3ab{\bf U}_4b^m$. Then either ${\bf U}_1$ or ${\bf U}_4$ contains a linear variable in $\bf U$ (otherwise $\bf U$ would had contained a subword  $a^{m+1}b$   and a subword   $ab^{m+1}$ which contradicts to the choice of $m$). Since all conditions are symmetric, without loss of generality we may assume  that the word ${\bf U}_1$ contains some variable $t$ linear in $\bf U$. Then by the choice of $m$, the word
$\bf U$ contains a subword $a^m {\bf U}_5 a^mb {\bf U}_2 b$ for some word ${\bf U}_5$ that contains $t$. Then ${\bf U} \preceq \{x^mtx, xtx^m\}$.
Consequently,   ${\bf U} \preceq  \{ x^{m}txyty,  ytyxtx^{m}\}$.

 $(ii) \rightarrow (iii)$  Suppose that
 the word $a^m{\bf U}_1ab{\bf U}_2b$ is not a subword of $\bf U$ for any ${\bf U}_1, {\bf U}_2 \in \mathfrak A ^*$.
Let $\Theta: \mathfrak A \rightarrow \mathfrak A^*$ be a substitution such that
 $\Theta(y) \Theta(x) \ne \Theta(x) \Theta(y)$.

 If  $\Theta(x)$ or $\Theta(y)$ is not a power of a variable then $\Theta(x^{m}txyty) \preceq xytxty$ or $\Theta(x^{m}txyty) \preceq xytxty$.
So, we may assume that  $\Theta(x)=a^l$ and $\Theta(y)=b^r$ for some $l,r>0$.
 Then $\Theta(x^{m}tyxty)$ is not a subword of $\bf U$.

 Since the word $a^m{\bf U}_1ab{\bf U}_2b$ is not a subword of $\bf U$ for any ${\bf U}_1, {\bf U}_2 \in \mathfrak A ^*$,
 the word $\Theta(x^{m}txyty )$ also is not a subword of $\bf U$.
Lemma \ref{lemmaOS} implies that the monoid $S(\{{\bf U}\})$ satisfies $x^{m}txyty \approx x^{m}tyxty$.

Implication $(iii) \rightarrow (i)$ is obvious. \end{proof}

\section{Two algorithms for recognizing finitely based words among words with at most two non-linear variables}

\begin{theorem} \label{main}
Let ${\bf U}$ be a word with at most two non-linear variables and $m$ be the maximum such that ${\bf U} \preceq x^my^m$. Then $S(\{{\bf U}\})$ is finitely based if and only if $\bf U$ satisfies each of the following conditions:

(i) At least two of the words $\{xytxty, xtytxy, xtxyty\}$ are not isoterms for $S(\{{\bf U}\})$;

(ii) If ${\bf U} \preceq xtxyty$ then $m>1$ and one of the words $\{ytyxtx^{m}, x^{m}txyty\}$ is not an isoterm for $S(\{{\bf U}\})$.
\end{theorem}

\begin{proof}
First, suppose that ${\bf U} \preceq \{xytxty, xtytxy\}$.
Then, in view of Corollary \ref{fb2} we may assume that  ${\bf U} \preceq \{xytxy, xytyx\}$.
Since $\bf U$ is a single word with two non-linear variables, the condition ${\bf U} \preceq \{xytxy, xytyx\}$ implies that  ${\bf U} \preceq xtxyty$.
Then by Corollary \ref{fb3}, the word $\bf U$ contains either
$ababa$ or $ab^{\beta}a$ for some $\beta>1$ as a subword.  Lemmas \ref{abba} and \ref{ababa} show that $\bf U$ is NFB in each of these cases.
So, if ${\bf U} \preceq \{xytxty, xtytxy\}$ then $\bf U$ is NFB.

Now suppose that ${\bf U} \not \preceq \{xytxty\}$ or ${\bf U} \not \preceq \{xtytxy\}$. If ${\bf U} \not \preceq xtxyty$ then $\bf U$ is hereditary finitely based
by Corollary \ref{11words}. If ${\bf U} \preceq xtxyty$ then by Corollary \ref{fb1l} either $\bf U$ is NFB or
${\bf U} \preceq xxyy$ and both words $\{xytxty, xtytxy\}$ are not isoterms for $S(\{{\bf U}\})$. In view of Theorem \ref{firstsim}, we have $S(\{{\bf U}\}) \models \{\sigma_1, \sigma_{2}\}$. Since ${\bf U} \preceq xxyy$, we have $m>1$. Consider two cases.

{\bf Case 1}: ${\bf U} \not \preceq \{ ytyxtx^{m}, x^{m}txyty \}$.

In this case, Lemma \ref{ident} implies that $S(\{{\bf U}\})$ satisfies either $x^{m}txyty \approx  x^{m}tyxty$ or $xtxyty^m \approx xtyxty^m$.
By Lemma \ref{oneword}, $S(\{{\bf U}\})$ satisfies Conditions (i) and (ii) in Lemma \ref{fbtlem}. Therefore, $S(\{{\bf U}\})$ is finitely based by Lemma \ref{fbtlem} or its dual.

{\bf Case 2}: ${\bf U} \preceq \{ ytyxtx^{m}, x^{m}txyty \}$.

In this case, Lemma \ref{oneword} (i) implies that
${\bf U} \preceq \{ ytyx^dtx^{m+1-d}, x^{m+1-d}tx^dyty | 0 <d \le m\}$ but one of the words $\{ x^{m+1}yty, ytyx^{m+1}\}$ is not an isoterm for $S$. Theorem \ref{t1}(vii) or its dual implies that $S(\{{\bf U}\})$ is non-finitely based.
\end{proof}

The next theorem gives us a computation-free way to recognize FB words among words with at most two non-linear variables.

\begin{theorem} \label{main2}
Let $\bf U$ be a word with at most two non-linear variables $a$ and $b$. Then the word $\bf U$  is finitely based if and only if $\bf U$ is either block-1-simple or there is a single block $\bf B$ for which ${\bf B} \not \in \{a^n, b^n | n \ge 0\}$ and this block satisfies one of the following conditions modulo renaming $a$ and $b$:

(i) ${\bf B} =ab^m$ for some $m > 0$ and $\bf B$ is the first non-empty block of $\bf U$;

(ii) ${\bf B} =b^ma$ for some $m > 0$ and $\bf B$ is the last non-empty block of $\bf U$;

(iii) ${\bf B}=a^{n}b^{m}$ such that  $min(n, m) =k>1$, $\bf U$ contains no $a$ to the right of $\bf B$,  no  $b$ to the left of $\bf B$
and one of the words $\{a {\bf U}_1ab {\bf U}_2b^k, a^k {\bf U}_1ab {\bf U}_2b\}$ is not a subword of ${\bf U}$ for any ${\bf U}_1, {\bf U}_2 \in \mathfrak A ^*$.

\end{theorem}

\begin{proof} According to the proof of Theorem \ref{main}, the word $\bf U$ is FB if and only if $\bf U$ is hereditary FB or $S(\{{\bf U}\}) \models \{\sigma_1, \sigma_2 \}$, ${\bf U} \preceq xxyy$ and ${\bf U} \not \preceq \{ ytyxtx^{k}, x^{k}txyty \}$.

In view of Corollary \ref{11words}, the word $\bf U$ is hereditary finitely based if and only if $\bf U$ is block-1-simple or there is a single  block $\bf B$ for which ${\bf B} \not \in \{a^n, b^n | n \ge 0\}$ and this block satisfies either Condition (i) or Condition (ii).

In view of Theorem \ref{firstsim} and Lemma \ref{ident}, $S(\{{\bf U}\}) \models \{\sigma_1, \sigma_2 \}$, ${\bf U} \preceq xxyy$ and ${\bf U} \not \preceq \{ ytyxtx^{k}, x^{k}txyty \}$ if and only if $\bf U$ satisfies Condition (iii).
\end{proof}

Following Definition 5.1 in \cite{MJ1}, we say that a word $\bf U$ is {\em hereditary finitely based} (HFB) if each subword of $\bf U$ is finitely based.
If the monoid $S(\{{\bf U}\})$ is hereditary finitely based, then evidently, the word $\bf U$ is HFB in the sense of Definition 5.1 in \cite{MJ1}.

\begin{cor} \label{hfbsnfb} A word $\bf U$ with at most two non-linear variables is FB if and only if $\bf U$ is HFB in the sense of Definition 5.1 in \cite{MJ1}.
\end{cor}

\begin{proof} According to Theorem \ref{main2}, if $\bf U$ is finitely based then either $\bf U$ is block-1-simple or $\bf U$ contains a single adjacent pair
$\{c, d\} \subseteq \ocs({\bf U})$ of occurrences of $a$ and $b$ such that either $\{c,d\} = \{{_{1{\bf u}}a}, {_{1{\bf u}}b} \}$ or $\{c,d\} = \{{_{last{\bf u}}a}, {_{last{\bf u}}b} \}$ or $\{c,d\} = \{{_{1{\bf u}}a}, {_{last{\bf u}}b} \}$. If  $\{c,d\} = \{{_{1{\bf u}}a}, {_{1{\bf u}}b} \}$ or $\{c,d\} = \{{_{last{\bf u}}a}, {_{last{\bf u}}b} \}$ then by Corollary \ref{11words}, the monoid $S(\{{\bf U}\})$ is hereditary finitely based, and consequently, the word $\bf U$ is HFB.

If $\{c,d\} = \{{_{1{\bf u}}a}, {_{last{\bf u}}b} \}$
then $\bf U$ satisfies Condition (iii) in Theorem \ref{main2}. Then each subword of $\bf U$ is either block-1-simple or also satisfies Condition (iii). In any case the word $\bf U$ is HFB.
\end{proof}

\begin{ex} Let  ${\bf U} =aat_1aabbt_2bb$. Then

 (i) the set $\{{\bf U}, a^4b^4\}$ is FB.

(ii) the word $\bf U$ is NFB.

(iii) each subword of $\bf U$ is FB.
\end{ex}

\begin{proof} (i) The set $\{{\bf U}, a^4b^4\}$ is finitely based by Lemma \ref{fbtlem} or by Theorem 4.4 in \cite{OS2}.

(ii) The word $\bf U$ is NFB  by Theorem \ref{main}  because  ${\bf U} \preceq \{xxtxyty, ytyxtxx\}$.

(iii) The word  ${\bf V}=aataabbtb$ is FB  by Theorem \ref{main2}  because the word $a{\bf U}_1ab{\bf U}_2bb$ is not a subword of $\bf V$ for any ${\bf U}_1, {\bf U}_2 \in \mathfrak A^*$. So, each subword of $\bf U$ is FB by symmetry and Corollary \ref{hfbsnfb}.
\end{proof}

\section{Seven NFB intervals between sets of block-2-simple words}

If $\Sigma$ is a set of identities and $L \subseteq {\mathfrak A}^*$  then $\iso(L, \Sigma)$ denotes the set of all words in $L$ that are isoterms for $\var \Sigma$.

\begin{fact}\label{last}
The set $W= \operatorname{Isot}(L, \Sigma)$ is the largest subset of $L$ such that $S(W)$ is contained in $\operatorname{var} \Sigma$.
\end{fact}

\begin{proof} Let $U$ be a subset of $L$ such that $S(U)$ is contained in $\var \Sigma$.
Then by Lemma \ref{prec} each word in $U$ is an isoterm for $\var \Sigma$. Therefore, $U$ is a subset of $\iso(L, \Sigma)$.
\end{proof}

Calculating $\iso(\mathfrak A^*, \Sigma)$ for certain sets of identities $\Sigma$ will be useful in the next article \cite{OS3}.
Theorem \ref{firstsim} immediately implies the following.

\begin{cor} \label{unik} For each $\Sigma \subseteq \{\sigma_1, \sigma_{\mu}, \sigma_2\}$
 the set  $\operatorname{Isot}(\mathfrak A^*, \Sigma)$ consists of all words $\bf u$ such that every adjacent pair of occurrences of two distinct non-linear variables in ${\bf u}$ is $\Sigma$-bad.
\end{cor}

Corollaries \ref{unik} and \ref{11words} immediately imply the following.

\begin{cor} \label{hfb}
A set of words is hereditary finitely based if and only if it is a subset of one of the following:

(i)  $\operatorname{Isot}(\mathfrak A^*, \sigma_1, \sigma_\mu)$ is the set of all words $\bf u$ such that every adjacent pair of occurrences of two non-linear variables $x \ne y$ in ${\bf u}$ is of the form $\{{_{last{\bf u}}x}, {_{last{\bf u}}y} \}$;

(ii) $\operatorname{Isot}(\mathfrak A^*, \sigma_2, \sigma_\mu)$ is the set of all words $\bf u$ such that every adjacent pair of occurrences of two non-linear variables $x \ne y$ in ${\bf u}$ is of the form $\{{_{1{\bf u}}x}, {_{1{\bf u}}y} \}$;

(iii) $\operatorname{Isot}(\mathfrak A^*, \sigma_1, \sigma_\mu, \sigma_2) = \operatorname{Isot}(\mathfrak A^*, \sigma_1, \sigma_\mu) \cap \operatorname{Isot}(\mathfrak A^*, \sigma_2, \sigma_\mu)$ is the set of all block-1-simple words.

\end{cor}

In view of Proposition 6.1 in \cite{OS2}, the monoid $S(\iso(\mathfrak A^*, \sigma_1, \sigma_\mu, \sigma_2))$ is finitely based by $\{ \sigma_1, \sigma_\mu, \sigma_2\}$.
However, the monoid $S(\iso(\mathfrak A^*, \sigma_1, \sigma_\mu))$ satisfies the identity $xytxy \approx xytyx$ (see Lemma 2.4 in \cite{OS3}) which does not follow from $\{ \sigma_1, \sigma_\mu \}$.

If $W_1 \subseteq W_2$ are sets of words then we use $[W_1, W_2]$ to refer to the interval  between $\var S(W_1)$ and
$\var S(W_2)$ in the lattice of all semigroup varieties. If $B2$ denotes the set of all block-2-simple words then
Theorem \ref{t1} immediately implies the following.

\begin{cor}  \label{sevenint} Every monoid in each of the following intervals is NFB:

(i) $[\{xytxy\}, \operatorname{Isot}(B2, xytyx \approx yxtxy)]$;

(ii) $[\{xytyx\}, \operatorname{Isot}(B2, xytxy \approx yxtyx)$;

(iii) $[\{xytxty, xtytxy\}, \operatorname{Isot}(B2, \{xytyx \approx yxtxy, xytxy \approx yxtyx\})]$;

(iv) $[\{xtxyty\}, \operatorname{Isot}(B2, \{xxyy \approx yyxx, xytxy \approx yxtyx\})]$;

(v) $[\{xxyy, xytxty\}, \operatorname{Isot}(B2, \{xytyx \approx yxtxy, xytxy \approx yxtyx \})]$;

(vi) $[\{xtxyty, xytxy, xytyx\}, \operatorname{Isot}(B2, \{xyyx \approx yxxy\} \cup  \{x^myxyx \approx yx^{m+2}y | m \ge 1\})]$;

(vii) $[\{xxyy\} \cup \{ytyx^dtx^{k-d}, x^{k-d}tx^dyty | 0 <d < k\}, \operatorname{Isot} (B2, \{x^kyty \approx yx^kty, xytxty \approx yxtxty, xtytxy \approx xtytyx\})]$, $k>2$.

\end{cor}

Using Corollary \ref{nfbcombinations}, one can easily check that if $I$ and $N$ are in the same row of Table \ref{classes} then for every set of block-2-simple words $W$ such that $W \preceq I$ and $W \not \preceq {\bf n}$ for any ${\bf n} \in N$ the monoid $S(W)$ belongs to the corresponding interval in Corollary \ref{sevenint}.

\subsection*{Acknowledgement} The author thanks Gili Golan for her thoughtful comments and Victor Guba for his suggestion to prove Theorem \ref{com} by using the argument of Lyndon and Schupp. The author is also very grateful to the anonymous referee for the helpful comments and suggestions.


\begin{thebibliography}{123}

\addcontentsline{toc}{section}{{\noindent\bf Bibliography}}

\bibitem{CHLS}   Chen, Y. Z., Hu, X., Luo Y. F., Sapir O. B.: \emph{The finite basis problem for the monoid of $2 \times 2$ upper triangular tropical matrices},
Bull. Aust. Math. Soc., {\bf 94}, 54--64 (2016)


\bibitem{MJ1} Jackson, M. G.: \emph{On the finite basis problem for finite Rees quotients of free monoids}, Acta. Sci. Math. (Szeged) {\bf 67},  121--159 (2001)


\bibitem{MJ} Jackson,  M. G.: \emph{Finiteness properties of varieties and the restriction to finite algebras}, Semigroup Forum, {\bf 70},  159--187 (2005)


\bibitem{JS}  Jackson,  M. G., Sapir, O. B.: \emph{Finitely based, finite sets of words}, Internat. J. Algebra Comput., {\bf 10}(6),  683--708 (2000)


\bibitem{GL} Lallement, G.: \emph{Semigroups and Combinatorial Applications}, John Wiley and Sons, N.Y.-Chichester-Brisbane-Toronto (1979)

\bibitem{EL1}  Lee, E. W. H.: \emph{Maximal Specht varieties of monoids}, Mosc. Math. J. {\bf 12}, 787--802 (2012)

\bibitem{LS} Lyndon R. C., Schupp P. E.: \emph{Combinatorial group theory}, Spriner-Verlag, Berlin-Heidelberg-N.Y. (1977)


\bibitem{RM}  McKenzie, R. N.: \emph{Tarski's finite basis problem is undecidable},  Internat. J. Algebra Comput., {\bf 6}, 49--104 (1996)


\bibitem{Pastin}  Pastijn, F.: \emph{Polyhedral convex cones and the equational
theory of the bicyclic semigroup}, J. Aust. Math. Soc. {\bf 81}, 63--96  (2006)



\bibitem{P} Perkins, P.: \emph{Bases for equational theories of semigroups},
  J. Algebra {\bf 11}, 298--314 (1969)

\bibitem{MS} Sapir, M. V.: \emph{Problems of Burnside type and the finite basis property in varieties of semigroups}, Math. USSR Izvestiya,
 {\bf 30}(2), 295--314 (1988)


\bibitem{MS1} Sapir, M. V.: \emph{Inherently nonfinitely based finite semigroups}, Math. USSR Sbornik,
 {\bf 61}(1), 155--166  (1988)


\bibitem{OS} Sapir, O. B.: \emph{Finitely based words}, Internat. J. Algebra Comput., {\bf 10}(4), 457--480 (2000)



\bibitem{OS1} Sapir, O. B.: \emph{Non-finitely based monoids}, Semigroup Forum, {\bf 90}(3), 557--586 (2015)


\bibitem{OS2} Sapir, O. B.: \emph{Finitely based monoids}, Semigroup Forum, {\bf 90}(3),  587--614 (2015)


\bibitem{OS3} Sapir, O. B.: \emph{Finitely based sets of 2-limited block-2-simple words}, Preprint, available under arXiv:1509.07920[math.GR]

 \bibitem{Sch} Schreier, O. J.: \emph{Die Untergruppen der freien Gruppen}, Abh. Math. Sem. Univ. Hamburg {\bf 5}, 161--183 (1927)


\bibitem{SV} Shevrin, L. N.,  Volkov, M. V.: \emph{Identities of semigroups}, Russian Math (Iz. VUZ), {\bf 29}(11), 1--64 (1985)


\bibitem{Sh} Shneerson,  L. M.: \emph{On the axiomatic rank of varieties generated by a semigroup or monoid with one defining relation}, Semigroup Forum {\bf 39}, 17--38 (1989)



\bibitem{MV} Volkov, M. V.: \emph{The finite basis problem for finite semigroups}, Sci. Math. Jpn.,  {\bf 53},  171--199 (2001)

\bibitem{LZL} Li, J. R., Zhang, W. T., Luo, Y. F.: \emph{On the finite basis problem for certain 2-limited words}, Acta Math. Sinica., {\bf 29}(3),  571--590 (2013)

\end{thebibliography}
\end{document}